\documentclass[dvips,ms]{MyStyle}

\usepackage{dsfont}
\usepackage{amsmath}
\usepackage{amssymb}
\usepackage{amsthm}
\usepackage{rotating}
\usepackage[bitstream-charter]{mathdesign}
\let\circledS\undefined

\usepackage{mathtools}
\usepackage{cancel}
\usepackage[version=3]{mhchem}
\usepackage{centernot}
\usepackage{centernotlarge}
\usepackage{xfrac}

\RequirePackage[OT1]{fontenc}
\RequirePackage[colorlinks,citecolor=blue,urlcolor=blue]{hyperref}
\usepackage{url}
\usepackage{booktabs,multirow}
\usepackage{rotating}
\usepackage[noadjust]{cite}
\usepackage{filecontents}

 \theoremstyle{plain}
 \newtheorem{theo}{Theorem}
 \newtheorem{lem}[theo]{Lemma}

 \theoremstyle{remark}
 \newtheorem{defi}[theo]{Definition}

 \newtheorem{rem}[theo]{Remark}

 \newcommand{\RR}{\mathds{R}}
 \newcommand{\NN}{\mathds{N}}
 \newcommand{\ZZ}{\mathds{Z}}
 
\newcommand{\E}{\operatorname{\mathds{E}}}
\newcommand{\Var}{\operatorname{\mathsf{Var}}}
\newcommand{\Cov}{\mathsf{Cov}}
\newcommand{\Sig}{\textrm{$\centernotlarge{\Sigma}$}}
 \newcommand{\ta}{\textrm{\rm t}}
 \newcommand{\m}{{\rm{m}}}
 \newcommand{\nm}{{\rm{nm}}}
 \newcommand{\bx}{\textrm{\bf\em x}}
 \newcommand{\bX}{\textrm{\bf\em X}}

 \newcommand{\varg}{g}

 \newcommand{\ds}{\displaystyle}

 \newcommand{\bm}[1]{\boldsymbol#1}

\let\originalleft\left
\let\originalright\right
\renewcommand{\left}{\mathopen{}\mathclose\bgroup\originalleft}
\renewcommand{\right}{\aftergroup\egroup\originalright}
\begin{document}

\begin{frontmatter}

\title{A note on a variance bound for the multinomial and the negative multinomial distribution}

\runtitle{}

\begin{aug}
\author{\fnms{Giorgos} \snm{Afendras}\thanksref{a,e1}\ead[label=e1,mark]{\footnotesize g\_afendras@math.uoa.gr},~}
\and
\author{\fnms{Vassilis} \snm{Papathanasiou}\thanksref{a,e2}\ead[label=e2,mark]{\footnotesize bpapath@math.uoa.gr}}

\address[a]{Department of Mathematics, Section of Statistics and O.R., University of Athens, Panepistemiopolis, 157 84 Athens, Greece\\
\printead{e1,e2}}

\runauthor{}

\affiliation{University of Athens}

\end{aug}

\begin{abstract}
We prove a Chernoff-type upper variance bound for the multinomial and the negative multinomial distribution. An application is also given.
\smallskip

\noindent
{\it AMS 2000 subject classifications:} Primary 60E15.
\end{abstract}

\begin{keyword}
\kwd{Cauchy-Schwartz inequality}
\kwd{Covariance identity}
\kwd{Variance bounds}
\end{keyword}

\end{frontmatter}

 \section{Introduction}
 \label{section introduction}
 Let $Z$ be a standard normal distribution and $\varg$ be an absolutely continuous function, with a.s.~ derivative $\varg'$. Chernoff \cite{Cher} proved that
 $\Var \varg(Z)\le\E\big(\varg'(Z)\big)^2$,
 provided that $\E\big(\varg'(Z)\big)^2$ is finite, where the equality holds iff $\varg$ is a linear polynomial; see also the previous papers by Nash \cite{Nash}, Brascamp and Lieb \cite{BL}.
 This inequality has been generalized and extended by many authors (see, e.g.,
 \cite{Chen,Cac,CP1,Cac2,CP2,Klaa,John,Pap,AB,HK,Hou,HP,APP1,APP2,CR,P-Rao,AP,AP2,Afe,WZ}).

 In discrete case, let $X$ be an integer-valued random variable with probability mass function (pmf) $p$ and finite mean $\mu$ and variance $\sigma^2$, and consider the function $w$ given by
 \[
 \sum_{j\le{x}}(\mu-j)p(j)=\sigma^2w(x)p(x)\quad\textrm{for all} \ \ x\in\ZZ.
 \]
 Then, for any suitable function $\varg$ the following inequality and Stein-type covariance identity hold (see Cacoullos and Papathanasiou \cite[Lemma 2.2]{CP1} and \cite[eq. (3.2)]{CP2})
 \begin{equation}\label{Descrete Chernoff}
 \Var\varg(X)\le\sigma^2\E w(X)[\Delta\varg(X)]^2,
 \end{equation}
 \begin{equation}\label{Descrete covariance identity}
 \Cov[X,\varg(X)]=\sigma^2\E w(X)\Delta\varg(X),
 \end{equation}
 where $\Delta$ is the forward difference operator (for the cases where $w$ is a quadratic polynomial see also Afendras et al.\ \cite{APP1,APP2}).

 Let now $\bm{X}=(X_1,\ldots,X_k)^\ta$ be a random vector with pmf supported by a ``convex'' set $C^k\subseteq\NN^k$ such that $\mathbf{0}\in{C^k}$ (``convex'' in the sense that if $\bm{x}=(x_1,\ldots,x_k)\in{C^k}$ then $\mbox{\Large\raisebox{-.25ex}{$\times$}}_{i=1}^{k}\{0,\ldots,x_i\}\subseteq{C^k}$). Assume that the mean $\bm{\mu}$ and the variance-covariance matrix $\Sig$ of $\bm{X}$ are well defined ($\Sig>0$) and consider the vector of linear functions
 \begin{equation}\label{vector q}
 \bm{q}(\bm{x})\equiv\big(q^1(\bm{x}),\ldots,q^k(\bm{x})\big)^\ta\coloneqq\Sig^{-1}\bm{x}.
 \end{equation}
 Then the $\bm{w}$-function of $\bm{X}$ is well defined for every $\bm{x}\in C^k$ by $\bm{w}(\bm{x})\equiv\big(w^1(\bm{x}),\ldots,w^k(\bm{x})\big)^\ta$ with
 \begin{equation}\label{w^i}
 w^i(\bm{x})p(\bm{x})=\sum_{j=0}^{x_i}\left[\mu^i-q^i(\bm{u}_i,j,\bm{v}_i)\right]p(\bm{u}_i,j,\bm{v}_i),
 \end{equation}
 where $\bm{u}_i=(x_1,\ldots,x_{i-1})$, $\bm{v}_i=(x_{i+1},\ldots,x_k)$ and $\mu^i=\E q^i(\bm{X})$, $i=1,\ldots,k$ (see \cite{CP3,PP}).
 Cacoullos and Papathanasiou \cite{CP3} extended the identity \eqref{Descrete covariance identity} as
 \begin{equation}\label{multivariate cov. ident.}
 \Cov[q^i(\bm{X}),\varg(\bm{X})]=\E w^i(\bm{X})\varg_i(\bm{X}),
 \end{equation}
 provided that $\E|w^i(\bm{X})\varg_i(\bm{X})|$ and $\E|(q^i(\bm{X})-\mu^i)\varg(\bm{X})|$ are finite, $i=1,2,\ldots,k$
 [for $\varg_i$ see Definition \ref{defi forward differences operators ect}(a) below];
 also, under the same conditions, they established the following inequality
 \begin{equation}
 \label{mult. lower variance bound C-P general}
 \Var\varg(\bm{X})\ge\E\big(w_1(\bm{X})\varg_1(\bm{X}),\ldots,w_k(\bm{X})\varg_k(\bm{X})\big)
                     \Sig
                     \E\big(w_1(\bm{X})\varg_1(\bm{X}),\ldots,w_k(\bm{X})\varg_k(\bm{X})\big)^\ta.
 \end{equation}
 If $\bm{X}$ is multinomial or negative multinomial distribution then the weight functions $w^i$ are the same, say $w$, and \eqref{mult. lower variance bound C-P general} takes the form
 \begin{equation}
 \label{mult. lower variance bound C-P}
 \Var\varg(\bm{X})\ge\E\big(w(\bm{X})\nabla^\ta\varg(\bm{X})\big) \Sig \E\big(w(\bm{X})\nabla\varg(\bm{X})\big),
 \end{equation}
 where $\nabla\varg$ is the discrete gradiant of $\varg$, see Definition \ref{defi forward differences operators ect}(b) below. This note complements this lower bound with the following upper bound:
 \[
 \Var\varg(\bm{X})\le \E\big(w(\bm{X}) \nabla^\ta\varg(\bm{X}) \Sig \nabla\varg(\bm{X})\big).
 \]
 Notice that for the continuous case of dependent random variables the similar bound has proven only in multivariate normal distribution by Chen \cite[eq.~(3.1)]{Chen}; also, Cacoullos \cite[eq.'s (1.1), (1.4)]{Cac2} generalize Chen's inequality for a vector of independent random variables, for both continuous and discrete cases (of course, Cacoullos's results cannot be apply for the multinomial and negative multinomial distributions, since both are vectors of dependent random variables).

  \section{Preliminaries}
 \label{section preliminaries}
 The following notations will be used in the sequel.
 \begin{defi}\label{defi forward differences operators ect}
 Let $y\in(-1,\infty)$, $\bm{x}=(x_1,\ldots,x_k)^\ta\in\RR^k$, $\bm{\pi}=(\pi_1,\ldots,\pi_k)^\ta\in(0,1)^k$ and $\varg:\RR^k\to\RR$. We denote by:
 \begin{itemize}
 \item[\rm(a)] $\varg_i(\bm{x})\equiv\Delta_i\varg(\bm{x})\coloneqq\varg(\bm{x}+\bm{e}_i)-\varg(\bm{x})$, where $\bm{e}_i$ is the $i$-th vector of the standard orthonormal basis of $\RR^k$.
 \item[\rm(b)] $\nabla^\ta\varg(\bm{x})\equiv\left(\nabla\varg(\bm{x})\right)^\ta\coloneqq(\varg_1(\bm{x}),\varg_2(\bm{x}),\ldots,\varg_k(\bm{x}))$.
 \item[\rm(c)] $\bm{\pi}^\bx\coloneqq\pi_1^{x_1}\cdots\pi_k^{x_k}$.
 \item[\rm(d)] ${y\choose \bx}\coloneqq\frac{\varGamma(y+1)}{x_1!\cdots x_k!\varGamma(y+1-x_1-\cdots-x_k)}$, provided that $\bm{x}\in\NN^k$ with $\sum_{i=1}^kx_i< y+1$.
 \item[\rm(e)] $\bm{x}_{-k}\coloneqq(x_1,\ldots,x_{k-1})^\ta\in\RR^{k-1}$.
 \end{itemize}
 \end{defi}
 \begin{defi}\label{defi b nb m nm}
 We shall use the following notations:
 \begin{itemize}
 \item[\rm(a)] $\m_k(n,\bm{\pi})$ the $k$-dimensional multinomial distribution with parameters $\bm\pi\in(0,1)^k$ and $n\in\NN$, namely with pmf $p(\bm{x})={n\choose{\bx}}\bm{\pi}^\bx\pi_0^{x_0}$, $\bm{x}\in\NN^k$ with $\sum_{i=1}^kx_i\le{n}$, where $x_0\coloneqq n-\sum_{i=1}^kx_i$ and $\pi_0\coloneqq1-\sum_{i=1}^k\pi_i>0$.
 \item[\rm(b)] $\nm_k(r,\bm{\theta})$ the $k$-dimensional negative multinomial distribution with parameters $\bm\theta\in(0,1)^k$ and $r>0$, namely with pmf $p(\bm{x})={r+\sum_{i=1}^kx_i-1\choose{\bx}}\bm{\theta}^\bx\theta_0^{r}$, $\bm{x}\in\NN^k$, where $\theta_0\coloneqq1-\sum_{i=1}^k\theta_i>0$.
 \item[\rm(c)] $p_k(\bm{x})\equiv p_{X_k}(x_k)$, $p_{-k}(\bm{x})\equiv p_{\bX_{-k}}(\bm{x}_{-k})$ and $p_{-k|k}(\bm{x})\equiv p_{\bX_{-k}|X_k=x_k}(\bm{x}_{-k})$ the pmf's of $X_k$, $\bm{X}_{-k}$ and $\bm{X}_{-k}|X_k=x_k$.
 \end{itemize}
 \end{defi}

 Now we present the $w$-functions of $\bm{X}$, $X_k$ and $\bm{X}_{-k}|X_k=x_k$, [$w(\bm{x})$, $w_k(\bm{x})$ and $w_{-k|k}(\bm{x})$, say], in both cases which we study.
 \begin{rem}
 \label{rem same w^i}
 For the multinomial and negative multinomial distributions the functions $w^i$ of \eqref{w^i} are the same for all $i$. Specifically, $w^i(\bm{x})=\pi_0^{-1}\big(n-\sum_{j=1}^{k}x_j\big)$ in $\m_k(n,\bm{\pi})$ and $w^i(\bm{x})=r^{-1}\theta_0\big(r+\sum_{j=1}^{k}x_j\big)$ in $\nm_k(r,\bm{\theta})$, see \cite[pp.\ 178--179]{CP3}. Note that we have corrected a minor misprint in the constant of the $w^i$ function corresponding to the negative multinomial.
 \end{rem}
 If $\bm{X}\sim\m_k(n,\bm{\pi})$ then $X_k\sim\m_1(n,\pi_k)$ and $\bm{X}_{-k}|X_k=x_k\sim\m_{k-1}\big(n-x_k,\frac{1}{1-\pi_k}\bm{\pi}_{-k}\big)$; so, we define
  \begin{equation}\label{w, w_k, w_-k|k multinomial}
 w(\bm{x})\coloneqq\frac{n-\sum_{i=1}^kx_i}{n\pi_0},
 \ \
 w_k(\bm{x})\coloneqq\frac{n-x_k}{n(1-\pi_k)}
 \ \ \textrm{and}\ \
 w_{-k|k}(\bm{x})\coloneqq\frac{(1-\pi_k)\left(n-\sum_{i=1}^kx_i\right)}{(n-x_k)\pi_0};
 \end{equation}
 noting that each function $h$ of $\bm{X}_{-k}|X_k=n$ is the zero constant of $\RR^{k-1}$ with probability $1$ [$\Var h=0$], so if $x_k=n$ then $w_{-k|k}$ is treated as zero-function. If $\bm{X}\sim\nm_k(r,\bm\theta)$ then $X_k\sim\nm_1\big(r,\frac{\theta_k}{\theta_0+\theta_k}\big)$ and $\bm{X}_{-k}|X_k=x_k\sim\nm_{k-1}(r+x_k,\bm{\theta}_{-k})$; so, the $w$-functions are defined by
 \begin{equation}\label{w w_k w_-k|k negative multinomial}
 \!\!w(\bm{x})\coloneqq\frac{\theta_0\left(r+\sum_{i=1}^kx_i\right)}{r},
 \
 w_k(\bm{x})\coloneqq\frac{\theta_0(r+x_k)}{r(\theta_0+\theta_k)}\\
 \ \textrm{and}\ 
 w_{-k|k}(\bm{x})\coloneqq\frac{(\theta_0+\theta_k)\left(r+\sum_{i=1}^kx_i\right)}{r+x_k}.
 \end{equation}

 For both cases one can easily see that
 \begin{equation}\label{useful}
 p_{-k|k}(\bm{x}+\bm{e}_k)=p_{-k|k}(\bm{x})w_{-k|k}(\bm{x})\\
 \quad\textrm{and} \quad
 w_{k}(\bm{x})w_{-k|k}(\bm{x})=w(\bm{x}).
 \end{equation}

 \begin{lem}\label{lem useful}
 Let $\bm{X}\sim\m_k(n,\bm\pi)$ or $\nm_k(r,\bm\theta)$ and consider a function $\varg$ such that $\E|X_j\varg(\bm{X})|$ and $\E|X_j\varg_i(\bm{X})|$ are finite for all $i,j=1,\ldots,k$. Then,

 \noindent
 {\rm(a)} the following covariance identity holds
 \begin{equation}\label{Cov(sum X_i, g)}
 \Cov\left[\sum_{i=1}^{k}X_i,\varg(\bm{X})\right]=\E\left(w(\bm{X})\sum_{i=1}^{k}c_{i}\varg_i(\bm{X})\right),
 \end{equation}
 where $w$ is given by {\rm\eqref{w, w_k, w_-k|k multinomial}} or {\rm\eqref{w w_k w_-k|k negative multinomial}}, respectively, and $c_{i}=\sum_{j=1}^{k}\sigma_{ij}$ with $\sigma_{ij}=\Cov(X_i,X_j)$;

 \noindent
 {\rm(b)} the next identity is valid (for the multinomial case only when $X_k<n$)
 \begin{equation}\label{Delta_1 E[g|X_k]}
 \Delta_k \E[\varg(\bm{X})|X_k]
 =\E\Bigg[w_{-k|k}(\bm{X})\Bigg(\varg_k(\bm{X})+\alpha_k\sum_{i=1}^{k-1}c_{i|k}\varg_i(\bm{X})\Bigg)\Bigg|X_k\Bigg],
 \end{equation}
 where $c_{i|k}=\sum_{j=1}^{k-1}\sigma_{ij|k}$ with $\sigma_{ij|k}=\Cov(X_i,X_j|X_k)$ and $a_k\equiv a(X_k)$ is $-\frac{1-\pi_k}{\pi_0(n-X_k)}$ or $\frac{\theta_0+\theta_k}{r+X_k}$.
 \end{lem}

 \begin{proof}
 (a) In view of Remark \ref{rem same w^i}, from \eqref{multivariate cov. ident.} we have that $\Cov[q^i(\bm{X}),\varg(\bm{X})]=\E w(\bm{X})\varg_i(\bm{X})$, $i=1,\ldots,k$. By \eqref{vector q} we get $\bm{X}=\Sig\bm{q}(\bm{X})$; so $\sum_{i=1}^kX_i=\sum_{i=1}^kc_iq^i(\bm{X})$. Combining the above relations \eqref{Cov(sum X_i, g)} follows.

 \noindent
 (b) Writing $\Delta_k\E[\varg(\bm{X})|X_k]=\E[\varg(\bm{X})|X_k+1]-\E[\varg(\bm{X})|X_k]$ and using \eqref{useful}, it follows that
 \[
 \begin{split}
 \Delta_k\E[\varg(\bm{X})|X_k]
 &=\E\left[w_{-k|k}(\bm{X})\varg_k(\bm{X})|X_k\right]
   +\E\left[w_{-k|k}(\bm{X})\varg(\bm{X})|X_k\right]-\E[\varg(\bm{X})|X_k]\\
 &=\E\left[w_{-k|k}(\bm{X})\varg_k(\bm{X})|X_k\right]+\Cov\left[w_{-k|k}(\bm{X}),\varg(\bm{X})|X_k\right],
 \end{split}
 \]
 since $\E\left[w_{-k|k}(\bm{X})|X_k\right]=1$ (see \cite[p.~178]{CP3}). In view of \eqref{w, w_k, w_-k|k multinomial} and \eqref{w w_k w_-k|k negative multinomial}, $w_{-k|k}(\bm{X})=\alpha_k\sum_{i=1}^{k-1}X_i+\beta_k$, where $\beta_k\equiv\beta(X_k)$ is a constant in $X_1,\ldots,X_{k-1}$; thus,
 \[
 \Delta_k\E[\varg(\bm{X})|X_k]=\E\left[w_{-k|k}(\bm{X})\varg_k(\bm{X})|X_k\right]+\alpha_k\Cov\Bigg[\sum_{i=1}^{k-1}X_i,\varg(\bm{X})|X_k\Bigg].
 \]
 Finally, from the conditions on $\varg$ it follows that $\E\big|X_j\varg(\bm{X})|X_k\big|$ and $\E\big|X_j\varg_i(\bm{X})|X_k\big|$ are finite for all $i,j=1,\ldots,k-1$. Thus, applying \eqref{Cov(sum X_i, g)} for $\bm{X}_{-k|k}$ the lemma is proved.
 \end{proof}

\section{The main result}
\label{section the main result}
 In this section we present the main result. An application in trinomial distribution is given.
 \begin{theo}\label{theo the main result}
 Let $\bm{X}\sim\m_k(n,\bm\pi)$ [or $\nm_k(r,\bm\theta)$] and consider a function $\varg$ such that $\Var\varg(\bm{X})<\infty$. Then,
 \begin{equation}\label{main result}
 \Var\varg(\bm{X})\le\E\left[w(\bm{X})\nabla^\ta\varg(\bm{X})\Sig\nabla\varg(\bm{X})\right],
 \end{equation}
 where $\Sig$ is the variance-covariance matrix of $\bm{X}$ and $w$ is given by {\rm\eqref{w, w_k, w_-k|k multinomial}} [or {\rm\eqref{w w_k w_-k|k negative multinomial}}]. The equality in {\rm\eqref{main result}} holds iff $\varg$ is a linear function with respect to $x_1,\ldots,x_k$, i.e.\ of the form $\varg(\bm{x})=\rho_0+\sum_{i=1}^{k}\rho_i x_i$.
 \end{theo}

 \begin{proof}
 If $\E\left[w(\bm{X})\nabla^\ta\varg(\bm{X})\Sig\nabla\varg(\bm{X})\right]=\infty$ then we have nothing to prove. Suppose that
 \begin{equation}
 \label{Suppose in proof of main}
 \E\big[w(\bm{X})\nabla^\ta\varg(\bm{X})\Sig\nabla\varg(\bm{X})\big]<\infty.
 \end{equation}
 The proof will be done by induction on $k$. For $k=1$  \eqref{main result} holds, see \eqref{Descrete Chernoff}. Assuming that \eqref{main result} is valid for $k-1$ for some $k>1$, we will prove that \eqref{main result} is also valid for $k$. It is well known that
 \begin{equation}\label{variance formula}
 \Var\varg(\bm{X})=\E[\Var(\varg(\bm{X})|X_k)]+\Var\left[\E(\varg(\bm{X})|X_k)\right].
 \end{equation}
 Using \eqref{Descrete Chernoff} for $X_k$ it follows that
 \begin{equation}\label{ineq Var E g|X_k}
 \Var\left[\E(\varg(\bm{X})|X_k)\right]\le\sigma_{k}^2\E w_k(\bm{X})\big(\Delta_k\E(\varg(\bm{X})|X_k)\big)^2,
 \end{equation}
 where $\sigma_{k}^2=\Var X_k$. From \eqref{Suppose in proof of main} we have that the conditions of Lemma \ref{lem useful} are valid; noting that ${w}_k(\bm{X})|_{X_k=n}=0$ with probability 1 and with help of \eqref{Delta_1 E[g|X_k]} we get
 \[
 \Var\left[\E(\varg(\bm{X})|X_k)\right]
 \le\sigma_{k}^2\E w_k(\bm{X})\bigg(\E w_{-k|k}(\bm{X})\bigg(\varg_k(\bm{X})
 +\alpha_k\sum_{i=1}^{k-1}c_{i|k}\varg_i(\bm{X})\bigg)\Big|X_k\bigg)^2.
 \]
 Since $\E [w_{-k|k}(\bm{X})|X_k]=1$, an application of Cauchy--Schwartz inequality gives 
 \begin{equation}\label{ineq C-S}
 \!\!
 \scalebox{.93}{
 $\ds\E^2\bigg[w_{-k|k}(\bm{X})\bigg(\varg_k(\bm{X})+\alpha_k\sum_{i=1}^{k-1}c_{i|k}\varg_i(\bm{X})\bigg)\Big|X_k\bigg] \le\E\bigg[w_{-k|k}(\bm{X})\bigg(\varg_k(\bm{X})+\alpha_k\sum_{i=1}^{k-1}c_{i|k}\varg_i(\bm{X})\bigg)^2\Big|X_k\bigg].$
 }
 \end{equation}
 Using \eqref{useful},
 \begin{align}
  \nonumber
 \scalebox{.93}{$\Var[\E(\varg(\bm{X})|X_k)]$}
  &\scalebox{.93}{$\ds\le\E\E\bigg[\sigma_k^2w(\bm{X})\bigg(\varg_k^2(\bm{X})+2\alpha_k\sum_{i=1}^{k-1}c_{i|k}\varg_i(\bm{X})\varg_k(\bm{X})
      +\bigg(\alpha_k\sum_{i=1}^{k-1}c_{i|k}\varg_i(\bm{X})\big)^2\bigg)\Big|X_k\bigg]$}\\
  &=\E\sigma_k^2w(\bm{X})\bigg(\varg_k^2(\bm{X})+2\alpha_k\sum_{i=1}^{k-1}c_{i|k}\varg_i(\bm{X})\varg_k(\bm{X})
    +\alpha_k^2\bigg(\sum_{i=1}^{k-1}c_{i|k}\varg_i(\bm{X})\bigg)^2\bigg).
  \label{E(Var)}
\end{align}
 By the induction hypothesis of \eqref{main result}, with $k-1$ in place of $k$, it follows that
 \begin{equation}\label{ineq Var g|X_k}
 \Var(\varg(\bm{X})|X_k)\le\E\big[w_{-k|k}(\bm{X})\nabla_{-k}^\ta\varg(\bm{X})\Sig_{-k|k}\nabla_{-k}\varg(\bm{X})|X_k\big],
 \end{equation}
 where $\Sig_{-k|k}$ is the variance-covariance matrix of $\bm{X}_{-k|k}$ and $\nabla_{-k}\varg=(\varg_1,\ldots,\varg_{k-1})^\ta$. Thus,
 \begin{align}
  \nonumber
    \E[\Var(\varg(\bm{X})|X_k)]
 &\le\E\E\big[w_{-k|k}(\bm{X})\nabla_{-k}^\ta\varg(\bm{X})\Sig_{-k|k}\nabla_{-k}\varg(\bm{X})|X_k\big]\\
  \nonumber
 &=\E w_{-k|k}(\bm{X})\nabla_{-k}^\ta\varg(\bm{X})\Sig_{-k|k}\nabla_{-k}\varg(\bm{X})\\
 \label{Var(E)}
 &=\E{w_{-k|k}}(\bm{X})\Bigg(\!\sum_{i=1}^{k-1}\!\sigma_{i|k}^2\varg_i^2(\bm{X})
   +2\!\sum_{1\le{i}<j\le{k-1}}\!\sigma_{ij|k}\varg_i(\bm{X})\varg_j(\bm{X})\Bigg).
 \end{align}
 From \eqref{variance formula}, via \eqref{E(Var)} and \eqref{Var(E)}, we get
 \[
 \begin{split}
 \Var\varg(\bm{X})\le\E\Bigg[&w(\bm{X})\sigma_k^2\varg_k^2(\bm{X})
  +\sum_{i=1}^{k-1}\Big[w(\bm{X})\sigma_k^2\alpha_k^2c^2_{i|k}+w_{-k|k}(\bm{X})\sigma_{i|k}^2\Big]\varg_i^2(\bm{X})\\
  &+2\sum_{i=1}^{k-1}w(\bm{X})\sigma_k^2\alpha_kc_{i|k}\varg_i(\bm{X})\varg_k(\bm{X})\\
  &+2\sum_{{1\le{i}<\atop{j}\le{k-1}}}\Big(w(\bm{X})\sigma_k^2\alpha_k^2c_{i|k}c_{j|k}+w_{-k|k}(\bm{X})\sigma_{ij|k}\Big)\varg_i(\bm{X})\varg_j(\bm{X})\Bigg].
 \end{split}
 \]
 After some algebra (see \cite{AfePap}), \eqref{main result} follows.

 Consider the function $\varg(\bm{x})=\rho_0+\sum_{i=1}^{k}\rho_i x_i$. One can easily see that \eqref{main result} holds as equality.
 Conversely, assume that \eqref{main result} holds as equality. Then \eqref{ineq Var E g|X_k}, \eqref{ineq C-S} and \eqref{ineq Var g|X_k} hold as equalities.
 From the equality in \eqref{ineq Var g|X_k}, under the inductional hypothesis, it follows that
 $\varg(\bm{x})=\varrho_0(x_k)+\sum_{i=1}^{k-1}\varrho_i(x_k) x_i$. From the equality in \eqref{ineq C-S} we have that the quantity
 $\varg_k(\bm{x})+\alpha_k\sum_{i=1}^{k-1}c_{i|k}\varg_i(\bm{x})$ is a constant in $x_1,\ldots,x_{k-1}$. Combining the above relations it follows that the quantity
 $\Delta_k\varrho_0(x_k)+\sum_{i=1}^{k-1}[\Delta_k\varrho_i(x_k)]x_i+\alpha_k\sum_{i=1}^{k-1}c_{i|k}\varrho_i(x_k)=\sum_{i=1}^{k-1}[\Delta_k\varrho_i(x_k)]x_i+h(x_k)$
 is a constant in $x_1,\ldots,x_{k-1}$. Therefore, $\Delta_k\varrho_i(x_k)=0$ for all $i=1,\ldots,k-1$, that is $\varrho_i(x_k)=\rho_i$, $i=1,\ldots,k-1$, are constants. Thus, $\varg(\bm{x})=\varrho_0(x_k)+\sum_{i=1}^{k-1}\rho_ix_i$.
 Finally, from the equality in \eqref{ineq Var E g|X_k} it follows that the quantity $\E(\varg(\bm{X})|X_k=x_k)$ is a linear function in $x_k$.
 Moreover, $\E(\varg(\bm{X})|X_k=x_k)=\E\big(\varrho_0(X_k)+\sum_{i=1}^{k-1}\rho_iX_i|X_k=x_k\big)=\varrho_0(x_k)+\sum_{i=1}^{k-1}\rho_i\E(X_i|X_k=x_k)$.
 For both cases the quantity $\sum_{i=1}^{k-1}\rho_i\E(X_i|X_k=x_k)$ is a linear function of $x_k$. Hence, $\varrho_0(x_k)$ is a linear function of $x_k$, i.e.\ $\varrho_0(x_k)=\rho_0+\rho_kx_k$, and the proof is complete.
 \end{proof}

 The present technique is based, mainly, on the fact that the functions $w^i$, $i=1,2,...,k$, are the same for multinomial and negative multinomial distributions, see Remark \ref{rem same w^i}. Of course, this is not true for all integer-valued multivariate distributions. Thus, in other cases, the present technique may not be applicable.

 \subsection{An application in negative trinomial distribution}
 \label{subsection an application}
 Next, we give an example in the trinomial distribution, in which the exact variance is rather difficult to compute, but the upper/lower bounds can be derived.

 Let $\bm{X}=(X_1,X_2)^\ta\sim\nm_2\big(1,\bm{\theta}=(\theta_1,\theta_2)^\ta\big)$, that is
 \[
 p(i,j)\equiv p_{X_1,X_2}(i,j)=\frac{(i+j)!}{i!j!}\theta_0\theta_1^i\theta_2^j,
 \quad i,j=0,1,\ldots,
 \]
 and consider the function $h(k)=1+\frac{1}{2}+\dots+\frac{1}{k}$, where $h$ is assumed to be zero if $k=0$. The statistic $T=T(\bm{X})=h(X_1)-h(X_2)$ is the unbiased estimator of $\ln\frac{1-\theta_2}{1-\theta_1}$, since
 \[
 \E T=\E h(X_1)-\E h(X_2)=-\ln\frac{\theta_0}{1-\theta_2}+\ln\frac{\theta_0}{1-\theta_1},
 \]
 see Afendras et al.\ \cite[pp.\ 180--181]{APP1}. The variance of $T$ is, clearly, quite complicated. However, the bounds of \eqref{mult. lower variance bound C-P} and \eqref{main result} can be used. Here $w(\bm{x})=\theta_0(x_1+x_2+1)$, $\Delta^\ta T=(T_1,T_2)=\big(\frac{1}{x_1+1},-\frac{1}{x_2+1}\big)$ and
 $\Sig=\theta_0^{-2}
       \Big(
        {\theta_1(1-\theta_2) \atop \theta_1\theta_2} \ \ {\theta_1\theta_2 \atop \theta_2(1-\theta_1)}
       \Big).$

 For the Cacoullos-Papathanasiou lower bound in \eqref{mult. lower variance bound C-P} we calculate $\E w(\bm{X})T_1=\frac{\theta_0}{1-\theta_2}$ and $\E w(\bm{X})T_2=-\frac{\theta_0}{1-\theta_1}$ (see \cite{AfePap}), and we get
 \begin{equation}
 \label{LB for T}
 \Var T>\frac{(1-\theta_1-\theta_2)(\theta_1+\theta_2)}{(1-\theta_1)(1-\theta_2)}.
 \end{equation}

 Applying \eqref{main result} we have that
 \[
 \Var T<\frac{\theta_1(1-\theta_2)}{\theta_0}E_1-2\frac{\theta_1\theta_2}{\theta_0}E_{12}+\frac{\theta_2(1-\theta_1)}{\theta_0}E_2,
 \]
 where $E_1=\E\frac{X_1+X_2+1}{(X_1+1)^2}$, $E_{1,2}=\E\frac{X_1+X_2+1}{(X_1+1)(X_2+1)}$ and $E_2=\E\frac{X_1+X_2+1}{(X_2+1)^2}$. After some algebra (see \cite{AfePap}) we found that $E_1=\frac{\theta_0}{\theta_1(1-\theta_2)}\ln\frac{1-\theta_2}{\theta_0}$, $E_{1,2}=\frac{\theta_0}{\theta_1\theta_2}\ln\frac{(1-\theta_1)(1-\theta_2)}{\theta_0}$ and $E_2=\frac{\theta_0}{\theta_2(1-\theta_1)}\ln\frac{1-\theta_1}{\theta_0}$; so,
 \begin{equation}
 \label{UB for T}
 \Var T<-\ln[(1-\theta_1)(1-\theta_2)].
 \end{equation}

 Table \ref{table} gives an idea on how the lower/upper bounds of $\Var T$ in \eqref{LB for T} and \eqref{UB for T} behave for various $\theta_1,\theta_2$-values, noting that both bounds are symmetric to $\theta_1$ and $\theta_2$.
 \begin{table}[htp]
 \caption{Numerical values of the upper/lower variance bounds given by \eqref{LB for T} and \eqref{UB for T} for parametric values $\theta_1,\theta_2=0.1,0.2,\ldots,0.8$, with $\theta_1+\theta_2<1$.}
 \label{table}
 \centering{\footnotesize
 \begin{tabular}
 {@{\hspace{0ex}}l@{\hspace{2.5ex}}c@{\hspace{2.5ex}}c@{\hspace{2.5ex}}c@{\hspace{2.5ex}}c@{\hspace{2.5ex}}c@{\hspace{2.5ex}}c@{\hspace{2.5ex}}c@{\hspace{2.5ex}}c@{\hspace{0ex}}}
 \addlinespace
 \toprule
  & \multicolumn{8}{c}{${\textrm{\!\!\!\bf upper bound:\quad}\boldsymbol{-\textrm{\bf ln}[(1-\theta_1)(1-\theta_2)]}\atop\textrm{\bf lower bound:\quad}\frac{\boldsymbol{(1-\theta_1-\theta_2)(\theta_1+\theta_2)}}{\boldsymbol{(1-\theta_1)(1-\theta_2)}}\quad }$}\\
 \cline{2-9}
 \raisebox{-0.4ex}{$\theta_1$} \quad \raisebox{0.1ex}{$\theta_2$} & $0.1$ & $0.2$ & $0.3$ & $0.4$ & $0.5$ & $0.6$ & $0.7$ & $0.8$  \\
 \midrule
 $0.1$ & ${0.211\atop0.198}$ & ${0.329\atop0.292}$ & ${0.462\atop0.381}$ & ${0.616\atop0.463}$ & ${0.799\atop0.533}$ & ${1.022\atop0.583}$ & ${1.309\atop0.593}$ & ${1.715\atop0.500}$\\
 [2ex]
 $0.2$ & ${0.329\atop0.292}$ & ${0.446\atop0.375}$ & ${0.580\atop0.446}$ & ${0.734\atop0.500}$ & ${0.916\atop0.525}$ & ${1.139\atop0.500}$ & ${1.427\atop0.375}$ & \\
 [2ex]
 $0.3$ & ${0.462\atop0.381}$ & ${0.580\atop0.446}$ & ${0.713\atop0.490}$ & ${0.868\atop0.500}$ & ${1.050\atop0.457}$ & ${1.273\atop0.321}$ &  & \\
 [2ex]
 $0.4$ & ${0.616\atop0.463}$ & ${0.734\atop0.500}$ & ${0.868\atop0.500}$ & ${1.022\atop0.444}$ & ${1.204\atop0.300}$ & & & \\
 \bottomrule
 \end{tabular}
 }
 \end{table}

\end{document}